\documentclass{amsart}

\newtheorem{theorem}{Theorem}[section]
\newtheorem{lemma}[theorem]{Lemma}
\newtheorem{corollary}[theorem]{Corollary}
\newtheorem{proposition}[theorem]{Proposition}
\theoremstyle{definition}
\newtheorem{definition}[theorem]{Definition}
\newtheorem{example}[theorem]{Example}

\theoremstyle{remark}

\numberwithin{equation}{section}

\begin{document}

\title{Measure-expansive systems}


%    author one information
\author{C. A. Morales}
\address{Beijing Advanced Innovation Center for Future Blockchain and Privacy Computing, Beihang University, Beijing, China \& Beijing Academy of Blockchain and edge Computing, Beijing, 100086.}
\curraddr{}
\email{morales@impa.br}
\thanks{}


%    author two information
%\author{}
%\address{}
%\curraddr{}
%\email{}
%\thanks{}

%    \subjclass is required.
%\subjclass[2000]{Primary }
%    The 2010 edition of the Mathematics Subject Classification is
%    now available.  If you are citing a classification from the
%    new scheme, use the following input coding instead.
% \subjclass[2010]{Primary }

\date{}

\dedicatory{}

%    Abstract is required.
\begin{abstract}
We call a dynamical system on a measurable metric space {\em measure-expansive}
if the probability of two orbits remain close each other for all time
is negligible (i.e. zero).
We extend results of expansive systems on compact metric spaces to the measure-expansive context.
For instance, the measure-expansive homeomorphisms are characterized as those homeomorphisms $f$ for which
the diagonal is almost invariant for $f\times f$ with respect to the product measure.
In addition, the set of points with converging semi-orbits for such homeomorphisms have measure zero.
In particular, the set of periodic orbits for these homeomorphisms is also of measure zero.
We also prove that there are no measure-expansive homeomorphisms in the interval and, in the circle, they are  the Denjoy ones.
As an application we obtain probabilistic proofs of some result of expansive systems.
We also present some analogous results for continuous maps.
\end{abstract}

\maketitle

\section{Introduction}
\label{intro}

\noindent
The {\em expansive homeomorphisms}, or, homeomorphisms
for which two orbits cannot remain close each other,
were introduced by Utz in the middle of the twenty century \cite{u} (see also \cite{gh}).
Since them an extensive literature about these homeomorphisms has been developed.

For instance, \cite{w1} proved
that the set of points doubly asymptotic to a given point for expansive homeomorphisms is at most countable.
Moreover, a homeomorphism of a compact metric space is expansive if
it does in the complement of finitely many orbits \cite{w2}.
In 1972 Sears proved the denseness of expansive homeomorphisms with respect to the uniform topology
in the space of homeomorphisms of a Cantor set \cite{se}.
An study of expansive homeomorphisms using generators is given in \cite{bw}.
Goodman \cite{g2} proved that every expansive homeomorphism of a compact metric space
has a (nonnecessarily unique) measure of maximal entropy and
Bowen \cite{b2} added specification to obtain unique equilibrium states.
In another direction, \cite{rr}
studied expansive homeomorphisms with canonical coordinates
and showed in the locally connected case that sinks or sources cannot exist.
Two years later, Fathi characterized expansive homeomorphisms on compact metric spaces as those exhibiting
adapted hyperbolic metrics \cite{f} (see also \cite{sa} or \cite{dgz} for more about adapted metrics).
Using this he was able to obtain an upper bound of the Hausdorff dimension and upper capacity
of the underlying space using the topological entropy.
In \cite{k} it is computed the large deviations of irregular periodic orbits for expansive homeomorphisms
with the specification property.
The $C^0$ perturbations of expansive homeomorphisms on compact metric spaces
were considered in \cite{cs}.
Besides, the multifractal analysis of expansive homeomorphisms with the specification property
was carried out in \cite{tv}.
We can also mention \cite{cz}
in which it is studied a new measure-theoretic pressure for expansive homeomorphisms.

From the topological viewpoint we can mention \cite{o} and \cite{r} proving the existence of expansive
homeomorphisms in the genus two closed surface, the $n$-torus and the open disk.
Analogously for compact surfaces obtained by making holes
on closed surfaces different from the sphere, projective plane and Klein bottle \cite{ka2}.
In \cite{ju} it was proved that there are no expansive homeomorphisms
of the compact interval, the circle and the compact $2$-disk.
The same negative result was obtained independently by Hiraide and Lewowicz in
the $2$-sphere \cite{h}, \cite{l}.
Ma\~n\'e proved in \cite{m}
that a compact metric space exhibiting expansive homeomorphisms
must be finite dimensional and, further, every minimal set of such
homeomorphisms is zero dimensional.
Previously he proved that the $C^1$ interior of the set of expansive diffeomorphisms
of a closed manifold is composed by pseudo-Anosov (and hence Axiom A) diffeomorphisms.
In 1993 Vieitez \cite{v1}
obtained results about expansive homeomorphisms on closed $3$-manifolds.
In particular, he proved that the denseness of the topologically hyperbolic periodic points
does imply constant dimension of the stable and unstable sets.
As a consequence a local product property is obtained for such homeomorphisms.
He also obtained
that expansive homeomorphisms on closed $3$-manifolds with dense topologically
hyperbolic periodic points are both supported on the $3$-torus and
topologically conjugated to linear Anosov isomorphisms \cite{v2}.

In light of these results it was natural to consider another notions of expansiveness.
For example, $G$-expansiveness, continuouswise and pointwise expansiveness were defined in \cite{da},
\cite{ka1} and \cite{r3} respectivelly.
We also have the entropy-expansiveness introduced by Bowen \cite{b}
to compute the metric and topological entropies in a large class of homeomorphisms.

In this paper we introduce a notion of expansiveness in which the Borel probability measures $\mu$ will
play fundamental role.
Indeed, we call a homeomorphism of a measurable metric space
{\em measure-expansive} (or {\em $\mu$-expansive} to indicate dependence on $\mu$) if
the probability of two orbits remain close each other for all time is zero.
%the probability that two orbits to be close one to another is zero.

It is clear that these homeomorphisms only exist for nonatomic measures and that, for such measures,
they include the expansive ones.
Besides, not every measure-expansive homeomorphism is
expansive and the identity is one which is entropy-expansive
but not measure-expansive.
We extend some result of the theory of expansive systems to the measure-expansive context.
For instance, measure-expansive homeomorphisms are characterized as those homeomorphisms $f$ for which
the diagonal is almost invariant for $f\times f$ with respect to the product measure.
In addition, the set of points with converging semi-orbits for such homeomorphisms have measure zero.
In particular, the set of periodic orbits for these homeomorphisms is also of measure zero.
We also prove that there are no measure-expansive homeomorphisms in any compact interval and, in the circle, we prove that they are precisely the Denjoy ones.
As an application we obtain probabilistic proofs of some result of expansive systems.
We also present some analogous results for continuous maps.

\section{Definition and examples}
\label{se2}

\noindent
In this section we introduce the definition of $\mu$-expansive homeomorphisms and present some examples.
To motivate let us recall the concepts of expansive and entropy-expansive
homeomorphisms \cite{u}, \cite{b}.

A homeomorphism $f: X\to X$ of a metric space $X$
is called {\em expansive}
if there is $\delta>0$ such that for every pair of different points $x,y\in X$ there is $n\in \mathbb{Z}$
such that $d(f^n(x),f^n(y))> \delta$.
Equivalently, $f$ is expansive if there is $\delta>0$ such that
$\Gamma_\delta(x)=\{x\}$ for all $x\in X$ where
$$
\Gamma_\delta(x)=\{y\in X:d(f^i(x),f^i(y)\leq \delta,\forall i\in \mathbb{Z}\}
$$
(when appropriated we write $\Gamma_\delta^f(x)$ to indicate dependence on $f$).
On the other hand, we call $f$ {\em entropy-expansive} if there is $\delta>0$ such that $h(f,\Gamma_\delta(x))=0$ for all $x\in X$
where $h(f,\cdot)$ denotes the topological entropy operation.

These definitions suggest further notions of expansiveness
related to a given property (P) of the closed sets in $X$.
More precisely, we say that $f$ is {\em expansive with respect to (P)} if
there is $\delta>0$ such that $\Gamma_\delta(x)$ satisfies (P) for all $x\in X$.
For example, a homeomorphism is expansive or h-expansive depending on whether
it is expansive with respect to the property of being a single point or
a zero entropy set respectively.
In this vein it is natural to consider the property of being {\em negligible in terms of some Borel probability measure $\mu$ of $X$}. This motivates the following definition.

\begin{definition}
A homeomorphism $f$ is {\em $\mu$-expansive} if there is $\delta>0$ such that $\mu(\Gamma_\delta(x))=0$ for all $x\in X$. The constant $\delta$ will be referred to as an {\em expansiveness constant} of $f$.
\end{definition}

Let us present some examples related to this definition.

\begin{example}
\label{ex1}
Clearly a measure $\mu$ for which there are $\mu$-expansive homeomorphisms must be
nonatomic. On the other hand, {\em if $\mu$ is nonatomic}, then every expansive homeomorphism is $\mu$-expansive.
\end{example}

\begin{example}
As is well known \cite{prv},
every complete separable metric space which either is uncountable or has no isolated points
exhibits nonatomic Borel probability measures.
It follows that every expansive homeomorphism in such a space is $\mu$-expansive for some
Borel probability $\mu$.
\end{example}

\begin{example}
There are expansive homeomorphisms on certain compact metric spaces which are not
$\mu$-expansive for all Borel probability measure $\mu$.
\end{example}

\begin{proof}
Consider the map $p(x)=x^3$ in $\mathbb{R}$
and define $X=\{0,1,-1\}\cup\{p^n(c):n\in \mathbb{N},c\in \{-\frac{1}{2},\frac{1}{2}\}\}$.
We have that $X$ is an infinite (but countable) compact metric space with the induced metric
$d(x,y)=|x-y|$.
Observe that there are no nonatomic Borel probability measures in $X$
since every non-isolated set of $X$ must be contained in
$\{-1,0,1\}$.
Defining $f(x)=p(x)$ for $x\in X$ we obtain an expansive homeomorphism $f$
which is not $\mu$-expansive for every Borel probability measure $\mu$.
\end{proof}

Further examples of homeomorphisms which are not $\mu$-expansive for all Borel probability measure $\mu$
can be obtained as follows. Recall that an {\em isometry} of a metric space $X$ is
a homeomorphism $f$ such that $d(f(x),f(y))=d(x,y)$ for all $x,y\in X$.

\begin{example}
 \label{isometry}
There are no $\mu$-expansive isometries of a separable metric space.
In particular, the identity map in these spaces (or the rotations in $\mathbb{R}^2$ or translations in $\mathbb{R}^n$) cannot be $\mu$-expansive for all $\mu$.
\end{example}

\begin{proof}   
Suppose by contradiction that there is a $\mu$-expansive isometry $f$ of a separable metric space
$X$ for some Borel probability measure $\mu$.
Since $f$ is an isometry we have $\Gamma_\delta(x)=B[x,\delta]$, where
$B[x,\delta]$ denotes the closed $\delta$-ball around
$x$. If $\delta$ is an expansivity constant of $f$, then $\mu(B[x,\delta])=\mu(\Gamma_\delta(x))=0$ for all $x\in X$.
Nevertheless, since $X$ is separable (and so Lindelof), we can select
a countable covering $\{C_1,C_2,\cdots, C_n,\cdots\}$ of $X$ by closed subsets
such that for all $n$ there is $x_n\in X$ such that $C_n\subset B[x_n,\delta]$.
Thus,
$\mu(X)\leq \sum_{n=1}^\infty\mu(C_n)\leq \sum_{n=1}^\infty \mu(B[x_n,\delta])=0$ which is a contradiction. This proves the result.
\end{proof}

\begin{example}
Endow $\mathbb{R}^n$ with a metric space with the Euclidean metric
and denote by $Leb$ the Lebesgue measure in $\mathbb{R}^n$.
Then, a linear isomorphism $f: \mathbb{R}^n\to \mathbb{R}^n$ is
$Leb$-expansive if and only if $f$ has eigenvalues of modulus less than or bigger than $1$.
\end{example}

\begin{proof} 
Since $f$ is linear
we have $\Gamma_\delta(x)=\Gamma_\delta(0)+x$ thus $Leb(\Gamma_\delta(x))=Leb(\Gamma_\delta(0))$ for all $x\in \mathbb{R}^n$ and $\delta>0$.
If $f$ has eigenvalues of modulus less than or bigger than $1$, then
$\Gamma_\delta(0)$ is contained in a proper subspace of $\mathbb{R}^n$ which implies
$Leb(\Gamma_\delta(0))=0$ thus $f$ is $Leb$-expansive.
\end{proof}

\begin{example}
As we shall see later, there are no $\mu$-expansive homeomorphism of a compact interval $I$
for all Borel probability measure $\mu$ of $I$.
In the circle $S^1$ these homeomorphisms are precisely the Denjoy ones.
\end{example}

Recall that a subset $Y\subset X$ is {\em invariant} if $f(Y)=Y$.

\begin{example}
\label{semilocal}
A homeomorphism $f$ is $\mu$-expansive, for some Borel probability measure $\mu$, if and only if
there is an invariant borelian set $Y$ of $f$ such that the restriction
$f/Y$ is $\nu$-expansive in $Y$ for some Borel probability measure $\nu$ of $Y$.
\end{example}

\begin{proof}
We only have to prove the only if part.
Assume that $f/Y$ is $\nu$-expansive in $Y$for some Borel probability measure $\nu$ of $Y$.
Fix $\delta>0$. Since $Y$ is invariant we have either $\Gamma_{\delta/2}^f(x)\cap Y=\emptyset$ or
$\Gamma_{\delta/2}^f(x)\cap Y\subset \Gamma_\delta^{f/Y}(y)$ for some $y\in Y$.
Therefore,
either $\Gamma_{\delta/2}^f(x)\cap Y=\emptyset$
or $\mu(\Gamma_{\delta/2}^f(x))\leq \mu(\Gamma_\delta^{f/Y}(y))$ for some $y\in Y$ where
$\mu$ is the Borel probability of $X$ defined by
$\mu(A)=\nu(A\cap Y)$. From this we obtain that for all $x\in X$ there is $y\in Y$ such that
$\mu(\Gamma_{\delta/2}^f(x))\leq\nu(\Gamma_\delta^{f/Y}(y))$.
Taking $\delta$ as an expansivity constant of $f/Y$
we obtain $\mu(\Gamma_{\delta/2}^f(x))=0$ for all $x\in X$
thus $f$ is $\mu$-expansive with expansivity constant $\delta/2$.
\end{proof}

The next example implies that $\mu$-expansiveness is invariant by conjugations.
Given a Borel measure $\mu$ in $X$ and a homeomorphism $\phi: X\to Y$ we denote by
$\phi_*(\mu)$ the pullback of $\mu$ defined by
$\phi_*(\mu)(A)=\mu(\phi^{-1}(A))$ for all borelian $A$.

\begin{example}
\label{conjugacy}
Let $\mu$ a Borel probability measure of $X$ and $f$ be a $\mu$-expansive homeomorphism.
If $\phi: X\to Y$ is a homeomorphism of compact metric spaces, then $\phi\circ f\circ \phi^{-1}$ is a $\phi_*(\mu)$-expansive
homeomorphism of $Y$.
\end{example}

\begin{proof}
Clearly $\phi$ is uniformly continuous so for all $\delta>0$
there is $\epsilon>0$ such that
$\Gamma_\epsilon^{\phi\circ f\circ \phi}(y)\subset \phi(\Gamma_\delta^f(\phi^{-1}(y)))$
for all $y\in Y$. This implies
$$
\phi_*(\mu)(\Gamma_\epsilon^{\phi\circ f\circ \phi}(y))\leq
\mu(\Gamma_\delta^f(\phi^{-1}(y))).
$$
Taking $\delta$ as the expansivity constant of $f$ we obtain that
$\epsilon$ is an expansivity constant
for $\phi\circ f\circ \phi^{-1}$.
\end{proof}

For the next example recall that a {\em periodic point} of a homeomorphism (or map)
$f: X\to X$ is a point $x\in X$ such that $f^n(x)=x$ for some $n\in \mathbb{N}^+$.
The nonwandering set of $f$
is the set $\Omega(f)$ formed by those points $x\in X$ such that
for every neighborhood $U$ of $x$ there is $n\in \mathbb{N}^+$ satisfying
$f^n(U)\cap U\neq\emptyset$.
Clearly a periodic point belongs to $\Omega(f)$ but not every point in $\Omega(f)$ is
periodic.
If $X=M$ is a {\em closed} (i.e. compact connected boundaryless) manifold and $f$ is a diffeomorphism we say that
an invariant set $H$ is {\em hyperbolic}
if there are a continuous invariant tangent bundle decomposition
$T_HM=E^s_H\oplus E^u_H$ and positive constants $K$, $\lambda>1$ such that
$$
\|Df^n(x)/E^s_x\|\leq K\lambda^{-n}
\quad \quad\mbox{ and }
\quad\quad
m(Df^n(x)/E^u_x)\geq K^{-1}\lambda^n,
$$
for all $x\in H$ and $n\in I\!\! N$ ($m$ denotes the co-norm operation in $M$).
We say that $f$ is {\em Axiom A} if $\Omega(f)$ is hyperbolic and the closure of the set of periodic points.

\begin{example}
Every Axiom A diffeomorphism with infinite nonwandering set of a closed manifold is $\mu$-expansive for some Borel probability measure $\mu$.
\end{example}

\begin{proof}
Consider an Axiom A diffeomorphism $f$ of a closed manifold.
It is well known that there is
a spectral decomposition $\Omega(f)=H_1\cup \cdots \cup H_k$
consisting of finitely many disjoint homoclinic classes $H_1,\cdots,H_k$ of $f$
(see \cite{hk} for the corresponding definitions).
Since $\Omega(f)$ is infinite we have that $H=H_i$ is infinite for some $1\leq i\leq k$.
As is well known $f/H$ is expansive.
On the other hand, $H$ is compact without isolated points since it is a homoclinic
class. It follows from Example \ref{ex1} that
$f/H$ is $\nu$-expansive for some Borel probability measure $\nu$ of $H$, so,
$f$ is $\mu$-expansive for some $\mu$ by Example \ref{semilocal}.
\end{proof}

\section{Equivalences}

\noindent
In this section we present some equivalences for $\mu$-expansiveness.
Hereafter all metric spaces $X$ under consideration will be compact unless otherwise stated.
We also fix a Borel probability measure $\mu$ of $X$.

To start we observe
an apparently weak definition of $\mu$-expansiveness saying that $f$ is
measure-expansive if there is
$\delta>0$ such that $\mu(\Gamma_\delta(x))=0$ {\em for $\mu$-almost every $x\in X$}. However, this definition
and the previous one are in fact equivalent by the
following lemma.

\begin{lemma}
\label{laguna1}
A homeomorphism $f$ is $\mu$-expansive
if and only if there is $\delta>0$ such that $\mu(\Gamma_\delta(x))=0$ for $\mu$-almost every $x\in X$.
\end{lemma}

\begin{proof}
We only have to prove the if part.
Let $\delta>0$ be such that $\mu(\Gamma_\delta(x))=0$ for $\mu$-almost every $x\in X$.
We shall prove that $\delta/2$ is a $\mu$-expansiveness constant of $f$.
Suppose by contradiction that it is not so.
Then, there is $x_0\in X$ such that $\mu(\Gamma_{\delta/2}(x_0))>0$.
Denote $A=\{x\in X:\mu(\Gamma_\delta(x))=0\}$
so $\mu(A)=1$.
Since $\mu$ is a probability measure we obtain
$A\cap \Gamma_{\delta/2}(x_0)\neq\emptyset$ so
there is $y_0\in \Gamma_{\delta/2}(x_0)$ such that $\mu(\Gamma_\delta(y_0))=0$.

Now we observe that since $y_0\in \Gamma_{\delta/2}(x_0)$ we have
$\Gamma_{\delta/2}(x_0)\subset \Gamma_\delta(y_0)$.
In fact, if $d(f^i(x),f^i(x_0))\leq \delta/2$ ($\forall i\in \mathbb{N}$)
one has
$d(f^i(x),f^i(y_0))\leq d(f^i(x),f^i(x_0))+d(f^i(x_0),f^i(y_0))\leq \delta/2+\delta/2=\delta$
($\forall i\in \mathbb{N}$) proving the assertion.
It follows that $\mu(\Gamma_{\delta/2}(x_0))\leq \mu(\Gamma_\delta(y_0))=0$ which is a contradiction.
This proves the result.
\end{proof}

In particular, we have the following corollary in whose statement supp$(\mu)$ denotes the support of $\mu$.

\begin{corollary}
A homeomorphism $f$ is $\mu$-expansive
if and only if there is
$\delta>0$ such that $\mu(\Gamma_\delta(x))=0$ for all $x\in$ supp$(\mu)$.
\end{corollary}

Another condition is as follows.
For every bijective map $f: X\to X$, $x\in X$, $\delta>0$ and $n\in \mathbb{N}^+$ we define
$$
V_f[x,\delta,n]=\{y\in X:d(f^i(y),f^i(y))\leq \delta,\mbox{ for all } -n\leq i<n\}.
$$
It is then clear that
$$
\Gamma_\delta(x)=\bigcap_{n\in \mathbb{N}^+}V_f[x,\delta,n]
$$
and $V_f[x,\delta,1]\supset V_f[x,\delta,2]\supset\cdots\supset V_f[x,\delta,n]\supset\cdots$
so we have
$$
\mu(\Gamma_\delta(x))=\lim_{n\to\infty}\mu(V_f[x,\delta,n])=\inf_{n\in \mathbb{N}^+}\mu(V_f[x,\delta,n])
$$
for all $x\in X$ and $\delta>0$.
From this we have the following lemma.

\begin{lemma}
 \label{suff-hom}
A homeomorphism $f$ is $\mu$-expansive
if and only if there is
$\delta>0$ such that
$$
\liminf_{n\to\infty}\mu(V_f[x,\delta,n])=0,
\quad\quad\mbox{ for all } x\in X.
$$
\end{lemma}

A direct application is the following measure-expansive version of
Corollary 5.22.1-(ii) of \cite{w}.

\begin{proposition}
\label{pp2}
Given $n\in \mathbb{Z}\setminus\{0\}$ a homeomorphism
$f$ is $\mu$-expansive if and only if $f^n$ is.
\end{proposition}

\begin{proof}
We can assume that $n>0$.
First notice that $V_f[x,\delta,n\cdot m]\subset V_{f^n}[x,\delta,m]$.
If $f^n$ is expansive then by Lemma \ref{suff-hom} there is $\delta>0$ such that for every $x\in X$ there is a sequence $m_j\to\infty$ such that
$\mu(V_{f^n}[x,\delta,m_j])\to 0$ as $j\to \infty$.
Therefore $\mu(V_f[x,\delta,n\cdot m_j])\to0$ as $j\to\infty$ yielding
$\liminf_{n\to\infty}\mu(V_f[x,\delta,n])=0$. Since $x$ is arbitrary we conclude that $f$ is positively $\mu$-expansive with constant $\delta$.

Conversely, suppose that $f$ is $\mu$-expansive with constant $\delta$. Since $X$ is compact
and $n$ is fixed we can
choose $0<\epsilon<\delta$ such that if $d(x,y)\leq\epsilon$, then $d(f^i(x),f^i(y))<\delta$
for all $-n\leq i\leq n$.
With this property one has
$\Gamma_\epsilon^{f^n}(x)\subset\Gamma_\delta^f(x)$ for all $x\in X$ thus $f^n$ is $\mu$-expansive with constant $\epsilon$.
\end{proof}

One more equivalence is motivated by a well known condition
for expansiveness stated as follows.

Given two metric spaces $X$ and $Y$ we always consider the product metric in $X\times Y$ defined by
$$
d((x_1,y_1),(x_2,y_2))=d(x_1,x_2)+d(y_1,y_2).
$$
If $\mu$ and $\nu$ are measures in $X$ and $Y$ respectively we denote by $\mu\times \nu$ their product measure in $X\times Y$.
If $f: X\to X$ and $g: Y\to Y$ we define their product
$f\times g:X\times Y\to X\times Y$,
$$
(f\times g)(x,y)=(f(x),g(y)).
$$
Notice that $f\times g$ is a homeomorphism if $f$ and $g$ are.
Denote by $\Delta=\{(x,x): x\in X\}$ the diagonal of $X\times X$.

Given a map $g$ of a metric space $Y$ we call an invariant set $I$ {\em isolated} if
there is a compact neighborhood $U$ of it such that
$$
I=\{z\in U:g^n(z)\in U,\forall n\in \mathbb{Z}\}.
$$
As is well known, a homeomorphism $f$ of $X$ is expansive if and only if
the diagonal $\Delta$ is an isolated set of $f\times f$ (e.g. \cite{ak}).
To express the corresponding measure-expansive version we introduce the following definition.
Let $\nu$ be a Borel probability measure of $Y$. We call an invariant set $I$ of $g$
{\em $\nu$-isolated} if there is a compact neighborhood $U$ of $I$
such that
$$
\nu(\{z\in Y:g^n(z)\in U,\forall n\in \mathbb{Z}\})=0.
$$
With this definition we have the following result
in which we write $\mu^2=\mu\times \mu$.

\begin{theorem}
A homeomorphism $f$ is $\mu$-expansive if and only if
the diagonal $\Delta$ is a $\mu^2$-isolated set of $f\times f$. 
\end{theorem}

\begin{proof}
Fix $\delta>0$ and a $\delta$-neighborhood
$U_\delta=\{z\in X\times X:d(z,\Delta)\leq \delta\}$ of $\Delta$.
For simplicity we set $g=f\times f$.

We claim that
\begin{equation}
\label{eqqq1}
\{z\in X\times X:g^n(z)\in U_\delta,\,\,\forall n\in \mathbb{Z}\}=\bigcup_{x\in X}(\{x\}\times \Gamma_\delta(x)).
\end{equation}
In fact, take $z=(x,y)$ in the left-hand side set.
Then,
for all $n\in \mathbb{Z}$ there is $p_n\in X$ such that
$d(f^n(x),p_n)+d(f^n(y),p_n)\leq \delta$
so $d(f^n(x),f^n(y))\leq\delta$ for all $n\in  \mathbb{Z}$ which implies
$y\in\Gamma_\delta(x)$. Therefore $z$ belongs to the right-hand side set.
Conversely, if $z=(x,y)$ is in the right-hand side set
then $d(f^n(x),f^n(y))\leq\delta$ for all $n\in  \mathbb{Z}$
so $d(g^n(x,y),(f^n(x),f^n(x)))=d(f^n(x),f^n(y))\leq \delta$ for all
$n\in  \mathbb{Z}$ which implies that $z$ belongs to the left-hand side set.
The claim is proved.

Let $F$ be the characteristic map of
the left-hand side set in (\ref{eqqq1}).
It follows that $F(x,y)=\chi_{\Gamma_\delta(x)}(y)$ for all $(x,y)\in X\times X$ where $\chi_A$ if the characteristic map of $A\subset X$.
So,
\begin{equation}
\label{eqqq2}
\mu^2(\{z\in X\times X:g^n(z)\in U_\delta,\,\,\forall n\in  \mathbb{Z}\})=
\int_X\int_X  \chi_{\Gamma_\delta(x)}(y)   d\mu(y)d\mu(x).
\end{equation}
Now suppose that $f$ is $\mu$-expansive with constant $\delta$.
It follows that
$$
\int_X  \chi_{\Gamma_\delta(x)}(y)   d\mu(y)=0,
\quad
\quad\forall x\in X
$$
therefore $\mu^2(\{z\in X\times X:g^n(z)\in U_\delta,\,\,\forall n\in  \mathbb{Z}\})=0$
by (\ref{eqqq2}).

Conversely, if $\mu^2(\{z\in X\times X:g^n(z)\in U_\delta,\,\,\forall n\in  \mathbb{Z}\})=0$ for some $\delta>0$,
then (\ref{eqqq2}) implies that $\mu(\Gamma_\delta(x))=0$
for $\mu$-almost every $x\in X$. Then, $f$ is $\mu$-expansive by Lemma \ref{laguna1}.
This ends the proof.
\end{proof}

Our final equivalence is given by using the idea of generators
(see \cite{w}).
Call a finite open covering $\mathcal{A}$ of $X$
{\em $\mu$-generator} of a homeomorphism $f$
if for every bisequence $\{A_n:n\in \mathbb{Z}\}\subset \mathcal{A}$
one has
$$
\mu\left(\bigcup_{n\in \mathbb{Z}}f^n(\mbox{Cl}(A_n))\right)=0.
$$
\begin{theorem}
\label{pp0}
A homeomorphism of $X$ is $\mu$-expansive if and only if it has a $\mu$-generator.
\end{theorem}

\begin{proof}
First suppose that $f$ is a $\mu$-expansive homeomorphism and let $\delta$ be its expansivity constant.
Take $\mathcal{A}$ as the collection of the open $\delta$-balls centered at $x\in X$.
Then, for any bisequence $A_n\in \mathcal{A}$ one has
$$
\bigcap_{n\in \mathbb{Z}}f^n(\mbox{Cl}(A_n))\subset\Gamma_\delta(x),
\quad\quad\forall x\in \bigcap_{n\in \mathbb{Z}}f^n(\mbox{Cl}(A_n)),
$$
so
$$
\mu\left(\bigcap_{n\in \mathbb{Z}}f^n(\mbox{Cl}(A_n))\right)\leq\mu(\Gamma_\delta(x))=0.
$$
Therefore, $\mathcal{A}$ is a $\mu$-generator of $f$.

Conversely, suppose that $f$ has a $\mu$-generator $\mathcal{A}$ and let
$\delta>0$ be a Lebesgue number of $\mathcal{A}$.
If $x\in X$, then for every $n\in \mathbb{Z}$ there is $A_n\in \mathcal{A}$ such that
the closed $\delta$-ball around $f^n(x)$ belongs to Cl$(A_n)$.
It follows that
$$
\Gamma_\delta(x)\subset\bigcap_{n\in \mathbb{N}}f^{-n}(\mbox{Cl}(A_n))
$$
so $\mu(\Gamma_\delta(x))=0$ since $\mathcal{A}$ is a $\mu$-generator.
\end{proof}

\section{Properties}
\label{sec4}

\noindent
In this section we present some properties of $\mu$-expansive homeomorphisms.
For this we introduce some basic notation.
Let $f: X\to X$ be a homeomorphism of a compact metric space $X$.
If $x,y\in X$, $n\in \mathbb{N}^+$ and $m\in \mathbb{N}$ we define
$$
A(x,y,n,m)=\{z\in X:
\max\{d(f^i(z),x),d(f^j(z),y)\}\leq \frac{1}{n},\forall i\leq-m\leq m\leq j\}
$$
and
$$
A(x,y,n)=\bigcup_{m\in \mathbb{N}}A(x,y,n,m).
$$
\begin{lemma}
\label{lemma-reddy0}
These sets satisfy the following properties:
\begin{enumerate}
 \item $A(x,y,n,m)$ is compact;
\item $A(x,y,n,m)\subseteq A(x,y,n,m')$ if $m\leq m'$;
\item $A(x,y,n',m)\subseteq A(x,y,n,m)$ and so $A(x,y,n')\subseteq A(x,y,n)$ if $n\leq n'$.
\end{enumerate}
\end{lemma}

Given $z\in X$ we define
$\omega(z)$ (resp. $\alpha(z)$) as the set of points $x=\lim_{k\to\infty}f^{n_k}(z)$ for some sequence $n_k\to\infty$ (resp. $n_k\to\infty$).
We say that $z\in X$ is a {\em point with converging semi-orbit under $f$}
if both $\alpha(z)$ and $\omega(z)$ consist of a unique point.
Denote by $A(f)$ the set of points with converging semi-orbits under $f$.
We say that $x\in X$ is a fixed point of $f$
if $f(x)=x$. Denote by Fix$(f)$ the set
of fixed points of $f$.

\begin{lemma}
 \label{lemma-reddy}
For every homeomorphism $f$ of a compact metric space $X$
there is as sequence $x_k\in$ Fix$(f)$ such that
\begin{equation}
\label{eq-reddy}
A(f)=\bigcap_{n\in \mathbb{N}^+}\bigcup_{k,k'\in \mathbb{N}}A(x_k,x_{k'},n).
\end{equation}
\end{lemma}

\begin{proof}
We have that Fix$(f)$ is compact since $f$ is continuous. It follows that
there is a sequence $x_k$ in Fix$(f)$
which is dense in Fix$(f)$. We shall prove that this sequence satisfies (\ref{eq-reddy}).

Take $z\in A(f)$. Then, there are $x,y\in X$ such that $\alpha(z)=x$ and $\omega(z)=y$.
Fix $n\in \mathbb{N}^+$. Then, there is
$m\in \mathbb{N}$ such that
$\max\{d(f^i(z),x),d(f^j(y),y)\}\leq \frac{1}{2n}$ whenever $i\leq -m\leq m\leq j$.
But clearly $x,y\in$ Fix$(f)$, so there are $k,k'\in \mathbb{N}$ such that
$\max\{d(x,x_k),d(y,x_{k'})\}\leq\frac{1}{2n}$.
Hence,
$$
\max\{d(f^i(z),x_k),d(f^j(z),x_{k'})\}\leq \frac{1}{n},\quad\quad\forall i\leq -m\leq m\leq j
$$
therefore
$z\in A(x_k,x_{k'},n,m)$. We have then proved that
for all $n\in \mathbb{N}^+$ there are $k,k'\in \mathbb{N}$ such that
$z\in A(x_k,x_{k'},n)$ thus $A(f)$ is contained in the right-hand side set of (\ref{eq-reddy}).

Conversely if $z$ belongs to the right-hand side set of (\ref{eq-reddy}), then
there are sequences $k_n,k'_n,m_n\in \mathbb{N}$ such that
$$
\max\{d(f^i(z),x_{k_n}),d(f^j(z),x_{k_n'})\}\leq \frac{1}{n},
\quad\quad\forall i\leq -m_n\leq m_n\leq j.
$$
By compactness there is a sequence $n_r\to\infty$ such that
$x_{k_{n_r}}\to x$ and $x_{k_{n_r}'}\to x'$ for some fixed points $x,x'$ of $f$.
We assert that $\alpha(z)=x$ and $\omega(z)=x'$.
Take $\epsilon>0$. Then, there is $r_1\in \mathbb{N}$ such that
$\frac{1}{n_r}\leq\frac{\epsilon}{2}$ and
$\max\{d(x_{k_{n_r}},x),d(x_{k'_{n_r}},x')\}\leq \frac{\epsilon}{2}$ for all $r\geq r_1$.
Then, for all $r\geq r_1$ and $i\leq -m_{n_{r}}\leq m_{n_{r}}\leq j$ one has
$d(f^i(z),x)\leq d(f^i(z),x_{k_{n_r}})+d(x_{k_{n_r}},x)\leq \frac{\epsilon}{2}+\frac{\epsilon}{2}=\epsilon$
and, further, $d(f^j(z),x')\leq d(f^j(z),x_{k_{n_r}'})+d(x_{k_{n_r}'},x')\leq \frac{\epsilon}{2}+\frac{\epsilon}{2}=\epsilon$
proving the assertion.

From this assertion we have $z\in A(f)$ then (\ref{eq-reddy}) holds.
\end{proof}

Hereafter we denote by $B[x,\delta]$ (resp. $B(x,\delta)$)
the closed (resp. open) $\delta$-ball of $X$ around $x$.

The following represents the measure-expansive version of a result in \cite{r}.

\begin{theorem}
\label{reddy}
If $\mu$ is a Borel probability measure of a compact metric space $X$ and
$f: X\to X$ is a $\mu$-expansive homeomorphism, then the set
of points with converging semi-orbits under $f$ has $\mu$-measure $0$.
\end{theorem}

\begin{proof}
Recall that $A(f)$ denotes the set of points with converging semi-orbits under $f$.
To prove $\mu(A(f))=0$ we assume by contradiction that $\mu(A(f))>0$.
By Lemma \ref{lemma-reddy} there is a sequence of fixed points $x_k$ of $f$
satisfying (\ref{eq-reddy}).
From this we obtain
\begin{equation}
\label{eq-reddy2}
\mu\left(\bigcup_{k,k'\in \mathbb{N}}A(x_k,x_{k'},n)\right)>0,
\quad\quad\forall n\in \mathbb{N}^+.
\end{equation}
Fix an expansivity constant $e$ of $f$. Fix
$n\in \mathbb{N}$ such that $\frac{1}{n}\leq \frac{e}{2}$.
Applying (\ref{eq-reddy2}) to this $n$ we can arrange
$k,k'\in \mathbb{N}$ such that
$\mu(A(x_k,x_{k'},n))>0$.
On the other hand, by Lemma \ref{lemma-reddy0}-(3), the definition of $A(x,y,n)$ and well known properties of measurable spaces we have
$$
\mu(A(x_k,x_{k'},n))=\sup_{m\in \mathbb{N}}A(x_k,x_{k'},n,m).
$$
Since $\mu(A(x_k,x_{k'},n))>0$, we can arrange $m\in \mathbb{N}$ satisfying
$$
\mu(A(x_k,x_{k'},n,m))>0.
$$
From this and the fact that
$A(x_k,x_{k'},n,m)$ is compact by Lemma \ref{lemma-reddy0}-(1)
we can select $z\in A(x_k,x_{k'},n,m)\cap$ supp$(\mu)$ and $\delta_z>0$ such that
$$
\mu(A(x_k,x_{k'},n,m)\cap B[z,\delta'])>0,
\quad\quad\forall 0<\delta'<\delta_z.
$$
But $f$ is continuous and the pair $(n,m)$ is fixed, so,
there is $0<\delta'<\delta_z$ such that
$$
d(f^i(z),f^i(w))\leq \frac{e}{2},
\quad\quad\forall -m\leq i\leq m, \forall w\in B[z,\delta'].
$$
Consider $w\in A(x_k,x_{k'},n,m)\cap B[z,\delta']$.
On the one hand, since $w\in B[z,\delta']$ we have
$d(f^i(w),f^i(z))\leq e$ for all $-m\leq i\leq m$
and, on the other, since $z,w\in A(x_k,x_{k'},n,m)$
and $\frac{1}{n}\leq\frac{e}{2}$ we have
$d(f^i(w),f^i(z))\leq d(f^i(w),x_k)+d(f^i(z),x_k)\leq e$
and $d(f^j(w),f^j(z))\leq d(f^j(w),x_{k'})+d(f^j(z),x_{k'})\leq e$
for all $i\leq -m\leq m\leq j$. This proves $w\in \Gamma_e(z)$ so
$$
A(x_k,x_{k'},n,m)\cap  B[z,\delta']\subset \Gamma_e(z).
$$
It follows that
$$
\mu(\Gamma_e(x))\geq \mu(A(x_k,x_{k'},n,m)\cap B[z,\delta'])>0
$$
which contradicts the $\mu$-expansiveness of $f$.
This ends the proof.
\end{proof}

A direct corollary is the following $\mu$-expansive version of Theorem 3.1 in \cite{u}.
Denote by Per$(f)$ the set of periodic points of $f$.

\begin{corollary}
\label{thA}
If $f$ is a $\mu$-expansive homeomorphism for some Borel probability measure $\mu$,
then $\mu($Per$(f))=0$.
\end{corollary}
 
\begin{proof}
Recalling Fix$(f)=\{x\in X:f(x)=x\}$ we have
Per$(f)=\cup_{n\in \mathbb{N}^+}$Fix$(f^n)$.
Now, $f^n$ is $\mu$-expansive by Proposition \ref{pp2}
and every element of Fix$(f^n)$ is a point with converging semi-orbits of $f^n$
thus $\mu($Fix$(f^n))=0$ for all $n$ by Theorem \ref{reddy}.
Therefore,
$\mu($Per$(f))\leq \sum_{n\in \mathbb{N}^+}\mu($Fix$(f^n))=0$.
\end{proof}

We finish this section by describing $\mu$-expansiveness in dimension one.
To start with we prove that there are no $\mu$-expansive homeomorphisms of compact intervals.

\begin{theorem}
\label{thD}
There are no $\mu$-expansive homeomorphisms of a compact interval $I$ for all Borel probability measure $\mu$ of $I$.
\end{theorem}

\begin{proof}
Suppose by contradiction that there is a $\mu$-expansive homeomorphism $f$ of $I$
for some Borel probability measure $\mu$ of $I$. Since
$f$ is continuous we have that Fix$(f)\neq \emptyset$.
Such a set is also closed since $f$ is continuous, so,
its complement $I\setminus$Fix$(f)$ in $I$ consists of countably many open intervals $J$.
It is also clear that every point in $J$ is a point with converging
semi-orbits therefore $\mu(I\setminus$ Fix$(f))=0$ by Theorem \ref{reddy}.
But $\mu($Fix$(f))=0$ by Corollary \ref{thA} so
$\mu(I)=\mu($Fix$(f))+\mu(I\setminus$ Fix$(f))=0$ which is absurd.
\end{proof}

Now we consider the circle $S^1$.
Recall that an orientation-preserving homeomorphism of the circle $S^1$ is {\em Denjoy}
if it is not topologically conjugated to a rotation \cite{hk}.

\begin{theorem}
\label{circle1}
A homeomorphism of $S^1$ is $\mu$-expansive for some Borel probability measure $\mu$ if and only if
it is Denjoy.
\end{theorem}

\begin{proof}
Let $f$ be a Denjoy homeomorphism of $S^1$.
As is well known $f$ has no periodic points and exhibits a unique minimal set
$\Delta$ which is a Cantor set \cite{hk}.
In particular, $\Delta$ is compact without isolated points thus it exhibits a nonatomic Borel probability meeasure $\nu$ (c.f. Corollary 6.1 in \cite{prv}).
On the other hand, one sees as in Example 1.2 of \cite{cl} that $f/\Delta$ is expansive so
it is $\nu$-expansive too.
Then, we are done by Example \ref{semilocal}.

Conversely, let $f$ be a $\mu$-expansive homeomorphism of $S^1$, for some $\mu$,
and suppose by contradiction that it is not Denjoy.
Then, either $f$ has periodic points or is conjugated to a rotation (c.f. \cite{hk}).
In the first case we can assume by Proposition \ref{pp2} that
$f$ has a fixed point.
Then, we can cut open $S^1$ along the fixed point to obtain a $\nu$-expansive
homeomorphism of $I$ for some Borel probability measure $\nu$ which contradicts Theorem \ref{thD}.
In the second case we have that
$f$ is conjugated to a rotation.
Since $f$ is $\mu$-expansive it would follow from
Example \ref{conjugacy} that there are $\nu$-expansive circle rotations for some Borel probabilities $\nu$.
However, such rotations cannot exist by Example \ref{isometry} since they are isometries.
This contradiction proves the result.
\end{proof}

In particular, there are no $C^2$ $\mu$-expansive diffeomorphisms of $S^1$ for all Borel probability measure $\mu$ of $S^1$. Similarly, there are no $C^1$ $\mu$-expansive diffeomorphisms of $S^1$ with derivative of bounded variation.

\section{Probabilistic proofs in expansive systems}

\noindent
The goal of this short section is to present the proof of some results in expansive systems
using the ours.

To start with we obtain another proof of the following result due to Utz (see Theorem 3.1 in \cite{u}).

\begin{corollary}
\label{utz1}
The set of periodic points of an expansive homeomorphism of a compact metric space is
countable.
\end{corollary}

\begin{proof}
Let $f$ be an expansive homeomorphism of a compact metric space $X$.
Since Per$(f)=\cup_{n\in \mathbb{N}^+}$Fix$(f^n)$
it suffices to prove that Fix$(f^n)$ is countable for all $n\in \mathbb{N}^+$.
Suppose by contradiction that Fix$(f^n)$ is uncountable for some $n$.
Since $f$ is continuous we have that Fix$(f^n)$ is also closed, so,
it is complete and separable with respect to the induced topology. Thus, by
Corollary 6.1 p. 210 in \cite{prv}, there is a nonatomic Borel probability measure
$\nu$ in Fix$(f^n)$.
Taking $\mu(A)=\nu(Y\cap A)$ for all borelian $A$ of $X$ we obtain
a nonatomic Borel probability measure $\mu$ of $X$ satisfying $\mu($Fix$(f^n))=1$.
Since Fix$(f^n)\subset$ Per$(f)$ we conclude that $\mu($Per$(f))=1$.
However, $f$ is expansive and $\mu$ is nonatomic so $f$ is $\mu$-expansive
thus $\mu($Per$(f))=0$ by Corollary \ref{thA} contradiction.
This contradiction yields the result.
\end{proof}

Next we obtain another proof of
the following result by Jacobson and Utz \cite{ju} (details in \cite{bry}).

\begin{corollary}
There are no expansive homeomorphisms of a compact interval.
\end{corollary}

\begin {proof}
Suppose by contradiction that there is an expansive homeomorphism of a compact interval $I$.
Since the Lebesgue measure $Leb$ of $I$ is nonatomic we obtain that $f$ is $Leb$-expansive.
However, there are no such homeomorphisms by Theorem \ref{thD}.
\end {proof}

The following lemma is motivated by the well known property that for every homeomorphism $f$ of a compact metric space $X$ one has that
supp$(\mu)\subset \Omega(f)$ for all $f$-invariant Borel probability measure $\mu$ of $X$.
Inded, we shall prove that this is true also for all $\mu$-expansive
homeomorphisms $f$ of $S^1$ even for non $f$-invariant measures $\mu$ of $S^1$.

\begin{lemma}
\label{supp}
If $\mu$ is a Borel probability measure of $S^1$, then supp$(\mu)\subset\Omega(f)$
for all $\mu$-expansive homeomorphism $f$.
\end{lemma}

\begin{proof}
Suppose by contradiction that
there is $x\in$ supp$(\mu)\setminus \Omega(f)$ for some
$\mu$-expansive homeomorphism $f$ of $S^1$. Let $\delta$ be an expansivity constant of $f$.
Since $x\notin\Omega(f)$ we can assume that
the collection of open intervals $f^n(B(x,\delta))$ as $n$ runs over $  \mathbb{Z}$ is disjoint.
Therefore, there is $N\in \mathbb{N}$ such that the length of $f^n(B(x,\delta))$
is less than $\delta$ for $|n|\geq N$.
From this and the continuity of $f$ we can arrange $\epsilon>0$
such that $B(x,\epsilon)\subset \Gamma_\delta(x)$
therefore $\mu(\Gamma_\delta(x))\geq \mu(B(x,\epsilon))>0$ as
$x\in$ supp$(\mu)$.
This contradicts the $\mu$-expansiveness of $f$ and the result follows.
\end{proof}

We use this lemma together with Theorem \ref{circle1} to obtain another proof
of the following result also by Jacobsen and Utz \cite{ju}.
Classical proofs can be found in Theorem 2.2.26 in \cite{ah},
Subsection 2.2 of \cite{cl}, Corollary 2 in \cite{r} and Theorem 5.27 of \cite{w}.

\begin{corollary}
There are no expansive homeomorphisms of $S^1$.
\end{corollary}

\begin {proof}
Suppose by contradiction that there is an expansive homeomorphism of $S^1$.
Since the Lebesgue measure $Leb$ of $S^1$ is nonatomic we obtain that $f$ is $Leb$-expansive.
It follows that
supp$(Leb)\subset \Omega(f)$ by Lemma \ref{supp}. However, $\Omega(f)$ is a Cantor set since
$f$ is Denjoy by Theorem \ref{circle1} and supp$(Leb)=S^1$ thus
we obtain a contradiction.
\end {proof}

\section{The map case}

\noindent
In this section we introduce the concept of positively $\mu$-expansive map
corresponding to that of $\mu$-expansive homeomorphisms.

First recall that a continuous map $f: X\to X$ of a metric space $X$ is
{\em positively expansive} (c.f. \cite{e}) if there is $\delta>0$ such that
for every pair of distinct points $x,y\in X$ there is $n\in \mathbb{N}$ such that
$d(f^n(x),f^n(y))>\delta$.
Equivalently, $f$ is positively expansive if there is $\delta>0$ such that
$\Phi_\delta(x)=\{x\}$ where
$$
\Phi_\delta(x)=\{y\in X:d(f^i(x),f^i(y))\leq \delta,\forall i\in \mathbb{N}\}
$$
(again we write $\Phi_\delta^f(x)$ to indicate dependence on $f$).
This motivates the following definition

\begin{definition}
A continuous map $f: X\to X$ is {\em positively $\mu$-expansive} if there is $\delta>0$ such that $\mu(\Phi_\delta(x))=0$ for all $x\in X$. The constant $\delta$ will be referred to as {\em expansiveness constant} of $f$.
\end{definition}

As in the homeomorphism case we have that
$f$ is positively $\mu$-expansive if and only if there is $\delta>0$ such that
$\mu(\Phi_\delta(x))=0$ for almost every $x\in X$.
Atomic measures $\mu$ do not exhibit positively $\mu$-expansive maps and, for the nonatomic $\mu$, every positively expansive map is positively $\mu$-expansive.
With the same argument as in the case of homeomorphisms
we can easily construct positively $\mu$-expansive maps which are not positively expansive.

An interesting question is motivated by the well known fact that every compact metric spaces supporting positively expanding homeomorphisms is finite \cite{s} (or \cite{rw} for another proof).
Indeed, we ask if the analogous result replacing expansive by $\mu$-expansive holds or not.
Actually, it seems that positively $\mu$-expansive homeomorphisms on compact metric spaces do not exist
(\footnote{Actually these homeomorphisms do exist.}).
One reason for this belief is that,
as in the case of homeomorphisms, we can prove that
if $X$ exhibits a positively $\mu$-expansive map then supp$(\mu)$ has no isolated points
(and so supp$(\mu)$ is infinite).

A necessary and sufficient condition for a given map to be positively $\mu$-expansive
is given as in the homeomorphism case.
Indeed, defining
$$
B_f[x,\delta,n]=\{y\in X:d(f^i(y),f^i(x))\leq\delta,\quad\forall 0\leq i<n\}
$$
we obtain
$$
\mu(\Phi_\delta(x))=\lim_{n\to\infty}\mu(B_f[x,\delta,n])=\inf_{n\in \mathbb{N}^+}\mu(B_f[x,\delta,n]),
\quad\quad\forall x\in X,\forall \delta>0,
$$
so, $f$ is positively $\mu$-expansive
if and only if there is
$\delta>0$ such that
\begin{equation}
\label{suff}
\liminf_{n\to\infty}\mu(B_f[x,\delta,n])=0,
\quad\quad\mbox{ for all } x\in X.
\end{equation}
It follows that for all $n\in \mathbb{N}^+$ a continuous map
$f$ is positively $\mu$-expansive if and only if $f^n$ is.
The proof is analogous to the corresponding result for homeomorphisms.

Another equivalent condition for positively $\mu$-expansiveness is given using the idea of positive generators
as in Lemma 3.3 of \cite{c}.
Call a finite open covering $\mathcal{A}$ of $X$
{\em positive $\mu$-generator} of $f$
if for every sequence $\{A_n:n\in \mathbb{N}\}\subset \mathcal{A}$
one has
$$
\mu\left(\bigcup_{n\in \mathbb{N}}f^n(\mbox{Cl}(A_n))\right)=0.
$$
As in the homeomorphism case we obtain the following proposition.
\begin{proposition}
\label{pp1}
A continuous map is $\mu$-expansive if and only if it has a positive $\mu$-generator.
\end{proposition}

We shall use this proposition to obtain examples of positively $\mu$-expansive maps.
If $M$ is a closed manifold (i.e. a compact connected boundaryless manifold) we call a differentiable map
$f: M\to M$ {\em volume expanding} if there are constants $K>0$ and $\lambda>1$ such that
$|$det$(Df^n(x))|\geq K\lambda^n$ for all $x\in M$ and $n\in \mathbb{N}$.
Denoting by $Leb$ the Lebesgue measure we obtain the following proposition.

\begin{proposition}
\label{exxx1}
Every volume expanding map of a closed manifold is positively $Leb$-expansive.
\end{proposition}

\begin{proof}
If $f$ is volume expanding there are $n_0\in \mathbb{N}$ and $\rho>1$ such that
$g=f^{n_0}$ satisfies $|\mbox{det}(Dg(x))|\geq \rho$ for all $x\in M$.
Then, for all $x\in M$ there is $\delta_x>0$ such that
\begin{equation}
 \label{equ1}
Leb(g^{-1}(B[x,\delta]))\leq \rho^{-1}Leb(B[x,\delta]),
\quad\quad\forall x\in M,\forall 0<\delta<\delta_x.
\end{equation}
Let $\delta$ be half of the Lebesgue number of the open covering
$\{B(x,\delta_x): x\in M\}$ of $M$.
By (\ref{equ1}) any finite open covering of $M$ by $\delta$-balls is
a positive $Leb$-generator, so, $g$ is positively $Leb$-expansive by Proposition \ref{pp1}.
Since $g=f^{n_0}$ we conclude that $f$ is positively $Leb$-expansive
(see the remark after (\ref{suff})).
\end{proof}

Again, as in the homeomorphism case,
we obtain an equivalent condition for positively $\mu$-expansiveness using
the diagonal.
Given a map $g$ of a metric space $Y$
and a Borel probability $\nu$ in $Y$ we say that
$I\subset Y$ is a {\em $\nu$-repelling set} if there is a neighborhood $U$ of $I$ satisfying
$$
\nu(\{z\in Y:g^{n}(z)\in U,\forall n\in \mathbb{N}\})=0.
$$
As in the homeomorphism case we can prove the following.

\begin{proposition}
A continuous map $f$ is positively $\mu$-expansive if and only if
the diagonal $\Delta$ is a $\mu^2$-repelling set of $f\times f$.
\end{proposition}

To finish we introduce an entropy allowing us to detect $\mu$-expansive maps.
To motivate it we recall the local entropy by Brin and Katok \cite{bk}.

The {\em local entropy} of $f$ with respect to
$\mu$ is the map $x\in X\mapsto h_\mu(f,x)$ defined by
$$
h_\mu(f,x)=\lim_{\delta\to 0^+}\limsup_{n\to\infty}-\frac{\log(\mu(B_f[x,\delta,n]))}{n}.
$$
Our entropy will be a variation of this definition. Consider
the map $\delta\mapsto e_\mu(f,\delta)$,
$$
e_\mu(f,\delta)=\inf_{x\in X}\limsup_{n\to\infty}-\frac{\log(\mu(B_f[x,\delta,n]))}{n}
$$
with the convention that $-\log 0=\infty$.
Clearly $e_\mu(f,\delta)$ increases as $\delta$ decreases to $0^+$ so
$\lim_{\delta\to 0^+}e_\mu(f,\delta)$
exists. We call this limit the {\em metric BK-entropy} of a $f$ with respect to $\mu$.
In other words,
$$
e_\mu(f)=\lim_{\delta\to 0^+}\inf_{x\in X}\limsup_{n\to\infty}-\frac{\log(\mu(B_f[x,\delta,n]))}{n}.
$$
This entropy has properties analogous to that of the classical metric entropy \cite{hk}.
For instance, $e_\mu(f^k)=ke_\mu(f)$ for all $k\in \mathbb{N}$ and $e_\mu(f)$ is invariant by
measure-preserving conjugacies.
An example with $e_\mu(f)=0$ is the identity map $I: X\to X$.
Examples with $e_\mu(f)>0$ are the $C^2$ Anosov diffeomorphisms on closed manifolds $M$
(with $\mu$ being in this case the Lebesgue measure of $M$).
This follows from the Bowen-Ruelle volume lemma \cite{br}.
It can be proved as well that
$e_\mu(f)=0$ for atomic measures $\mu$ therefore one can apply the Brin-Katok Theorem \cite{bk} and the classical
variational principle \cite{d}, \cite{g}, \cite{w} to obtain the inequality
$$
\sup_{\mu\in\mathcal{M}_f(X)} e_\mu(f)\leq h(f),
$$
where $h(f)$ is the topological entropy and $\mathcal{M}_f(X)$ is the space of Borel probability invariant measures of $f$.

Our interest by $e_\mu(f)$ is given below.

\begin{theorem}
Every continuous map $f$ for which $e_\mu(f)>0$ is positively $\mu$-expansive
\end{theorem}

\begin{proof}
Since $e_\mu(f)>0$ there are $\delta>0$, $\rho>0$ and $c>0$ such that for
every $x\in X$ there is a sequence $n_k^x\to\infty$ satisfying $\mu(B_f[x,\delta,n_k^x])\leq c e^{-\rho n^x_k}$ for all $k\in \mathbb{N}$.
Since $\rho>0$ we have that
$\mu(B_f[x,\delta,n^x_k])\to 0$ as $k\to\infty$ so
$\liminf_{n\to\infty}\mu(B_f[x,\delta,n])=0$ for all $x\in X$.
Then, the result follows from (\ref{suff}).
\end{proof}


\bibliographystyle{amsplain}

\end{document}